%% file: main.tex
\documentclass[11pt,english]{article}
\usepackage{lmodern}

\usepackage[T1]{fontenc}
\usepackage[utf8]{inputenc}
\usepackage{units}
\usepackage{dsfont}
\usepackage{amsmath}
\usepackage{amsthm}
\usepackage{amssymb}
\usepackage{geometry}
\makeatletter
\numberwithin{figure}{section}
\numberwithin{equation}{section}
\newcommand{\prob}{\mathrm{Pr}}
\@ifundefined{date}{}{\date{}}
\usepackage{float,psfrag,epsfig,color,url,hyperref}
\usepackage{algorithm,algorithmic}
\usepackage{graphicx,relsize}
\usepackage{amssymb,amsfonts,amsmath,amsthm,amscd,dsfont,mathrsfs,mathtools,microtype,nicefrac,pifont}
\usepackage{upgreek}
\usepackage[dvipsnames]{xcolor}
\usepackage{epstopdf,bbm,enumitem}
\usepackage{dsfont,tikz}
\usepackage[mathscr]{euscript}
\usepackage[toc,page]{appendix}
\hypersetup{
  colorlinks,
  linkcolor={red!50!black},
  citecolor={blue!50!black},
  urlcolor={blue!80!black}
}
\usepackage{etoolbox}
\patchcmd{\thebibliography}
  {\settowidth}
  {\setlength{\itemsep}{0pt plus 0.1pt}\settowidth}
  {}{}
\apptocmd{\thebibliography}
  {\small}
  {}{}
\allowdisplaybreaks
\usepackage{custom}
\makeatletter
\renewcommand{\paragraph}{\@startsection{paragraph}{4}{\z@}{1.25ex \@plus 1ex \@minus .2ex}{-1em}{\normalfont\normalsize\bfseries}}
\makeatother

\makeatother

\theoremstyle{plain}
\newtheorem{thm}{\protect\theoremname}[section]

\newtheorem{lem}[thm]{\protect\lemmaname}
\usepackage{babel}
\providecommand{\lemmaname}{Lemma}
\providecommand{\propositionname}{Proposition}
\providecommand{\theoremname}{Theorem}

\begin{document}
\input{math-macros.tex}

\title{A Simpler Proof of the Matrix Spencer Theorem}
\author{Nikhil Bansal\\ University of Michigan\\ \texttt{bansaln@umich.edu} \and Yunbum Kook\\ University of Michigan\\  \texttt{ybkook@umich.edu}}
\maketitle
\begin{abstract}
We give a simple exposition of the Matrix Spencer theorem
due to Akbas and Sra~\cite{AkbasSra26}.    
\end{abstract}

\section{Introduction \label{sec:introduction}}
For a matrix $A$, let $\norm{A}$ denote  its operator norm. 
Recently, Akbas and Sra \cite{AkbasSra26} proved the following remarkable result, generalizing the classical result of Spencer \cite{Spencer85}.  
\begin{thm}[Matrix Spencer]
 \label{thm:spencer} For any real symmetric matrices $A_{1},\ldots,A_{d}\in\Rdd$
with $\norm{A_{i}}\le1$, there are signs $\veps_{1},\ldots,\veps_{d}\in\{-1,1\}$
satisfying $
\norm{\sum^{d}_{i=1}\veps_{i}A_{i}}\lesssim d^{1/2}$.
\end{thm}

The proof of \cite{AkbasSra26} is based on a clever barrier function and interpolation argument, together with a matrix Poincar\'e inequality obtained by the Bochner technique, and is quite technical. We give a substantially shorter analysis using a simpler potential and a recent matrix Brascamp--Lieb inequality due to Cordero-Erausquin \cite{CorderoErausquin18}. While our proof gives weaker results for some intermediate steps, it suffices for Theorem \ref{thm:spencer}, and allows us to highlight the key ideas of~\cite{AkbasSra26} cleanly. 

\paragraph{AI disclosure.} Our proof was obtained after multiple iterations with GPT 6. The exposition is entirely by the authors, and they take full responsibility for the content and correctness.

\paragraph{Acknowledgment.} The authors thank all participants in the reading group on the matrix Spencer theorem at the Simons Institute.
NB was supported in part by NSF Awards CCF-2327011 and CCF-2504995.
YK is supported in part by NSF Awards CCF-2504995 and CCF-2236669.
\section{Partial Coloring}
\label{sec:general}

We use the standard partial coloring method. Initially all coordinates
are zero. In each round, a constant fraction of the remaining (active) coordinates are assigned values in $\{-1,1\}$. When $n \leq d$ active coordinates remain, we will show a partial assignment with discrepancy $O(d^{1/3} n^{1/6}) \lesssim d^{1/2} (n/d)^{1/6}$.\footnote{
Akbas and Sra \cite{AkbasSra26} show a substantially better bound of $O(n^{1/2} \log^2(d/n))$, using more involved calculations.} 
As $(n/d) = 1$ initially, and it shrinks geometrically in each round, this gives a full $\pm 1$ assignment after $O(\log n)$ rounds with discrepancy $\O(d^{1/2})$, proving Theorem \ref{thm:spencer}.

Let $L\gg 1$ be a sufficiently large constant. Henceforth, we fix the number of matrices $n$ and their size $d$, and our goal is to show a partial coloring bound of $r:=L d^{1/3} n^{1/6}$.

\paragraph{Gaussian Volume Estimate.}
Let $\gamma_{n}$ denote standard Gaussian measure on $\Rn$.
By a standard argument\footnote{\cite{Giannopoulos97} shows that any origin-symmetric convex body $\K$ in $\R^n$ with $\gamma_n(\K) \geq \exp(-n/10)$, contains a point $x \in \{-1,0,1\}^n$ with at least $n/4$ nonzeros.
Now, apply this to $\K = \{x: \norm{\sum_i x_i A_i} \leq r\}$.}  of Giannopoulos~\cite{Giannopoulos97}, to show the partial coloring, it suffices to show that 
\begin{equation}
\label{eq1}\prob_{g\sim \gamma_n} \Big[\Bnorm{\sum_{i} g_i A_i} \leq  r \Big]  \geq e^{-n/10}\,.
\end{equation}
Let us rescale and define the random Gaussian matrix $X = X(g):=\frac{1}{r}\,\sum^{n}_{i=1}g_{i}A_{i}.$
Then \eqref{eq1} is same as  
\begin{equation}
    \label{eq2}
  \prob_{\gamma_n}[\norm{X}<1] =  \E_{\gamma_n}[\ind_{\{\norm{X}<1\}}] \geq e^{-n/10}\,.
\end{equation}

\subsection{Modified Logarithmic Barrier}\label{sec:modified-barrier}
As the indicator  function
$\ind_{\{\norm{X}<1\}}$
is non-smooth, we will work with a smooth proxy. 
A natural candidate is $e^{-L(X)}$ where  $L(X)$ is the log-barrier, defined as $-\Tr \log (I-X^2)$ for $\|X\|<1$ and $+\infty$ outside $\{\norm X<1\}$. Note also that $-\Tr \log (I-X^2) = -\log \det (I-X^2)$.

As \[-\log (I-X^2) = X^2+ X^4/2 + X^6/3 + \cdots \qquad \text{for } \|X\|< 1\,,\]
we have $L(X)\geq 0$ for $\|X\|<1$, and thus $e^{-L(X)} \leq  \ind_{\{\norm{X}<1\}}$ for all $X$. 
Thus, to prove \eqref{eq2} it suffices to show that $\E_{\gamma_n}e^{-L(X)} \ge e^{-n/10}$.
Unfortunately, the proof below only gives $\E_{\gamma_n}e^{-L(X)} \geq \exp(-\Omega(d))$, which is too weak for us when $n\ll d$.

Hence, we instead consider a modification of $L(X)$, where we decrease the weight on $X^2$ and increase the weight on the other terms (we give some intuition for this later).\footnote{Akbas and Sra~\cite{AkbasSra26} reweight several Taylor coefficients,
to show the sharper bound $O(\sqrt{n}\,(1+\log(d/n))^{2})$.}

For $\norm X<1$, let $s(X):=-\log(I-X^{2})-X^{2}= X^{4}/2+ X^{6}/{3}+\cdots$. 
We define 
\begin{equation}
V(X) := b\Tr(X^{2})+\alpha\Tr s(X)\,,\label{eq:general-potential}
\end{equation}
and set $V(X)=+\infty$ outside $\{\norm X<1\}$. We choose (in hindsight)
\begin{equation}
b=16\sqrt{\eta}\,,\qquad\alpha=16/\sqrt{\eta}  \qquad  \text{where }
\eta:=n/r^{2}=L^{-2} (n/d)^{2/3}\,. \label{eq:general-parameters}
\end{equation}
Note that $\eta\ll 1$ (as $n\leq d$  and $L\gg 1$), and it shrinks as $n$ decreases. Intuitively, $V(X)$ still blows up as $\|X\|$ approaches $1$, but has smaller dependence on $X^2$.
Also $V$ is nonnegative,
convex, and even. Thus as before,
$e^{-V(X)} \leq \ind_{\{\norm X<1\}}$, and we focus henceforth on showing that
\begin{equation}
    \label{eq3}
    \E_{\gamma_{n}} e^{-V(X)}\geq e^{-n/10}\,.
\end{equation}

\subsection{Interpolation}
A key idea will be to consider the linear interpolation $X_t = tX$, and track how $\E_{\gamma_{n}} e^{-V(X_t)}$ evolves as $t$ goes from $0$ to $1$. Let us define
\[
Z_{t}=\E_{\gamma_{n}}e^{-V(X_t)}\qquad\text{and}\qquad F(t)=\log Z_t\,.
\]
Then \eqref{eq3} is same as showing $F(1)\geq -n/10$. 

As $F(0)=0$ (as $X_0(g)=0$ for all $g$), it suffices to show that $F'(t) \geq -n/10$ for all $0<t\le1$, so let us compute $F'(t)$.

\begin{lem}
\label{lem:F'(t)} For $0<t\le1$, we have $-F'(t)=\frac{2}{t}\,\{b\Tr(\E_{\mu_{t}}[X^{2}_{t}])+\alpha\Tr(\E_{\mu_{t}}[X^{4}_{t}(I-X_t)^{-1}])\}$, where
\[\D\mu_{t}(g)=\frac{1}{Z_{t}}\,e^{-V(X_{t}(g))}\,\D\gamma_{n}(g)\,.\]
As $V$ is convex, $\mu_t$ is a log-concave tilt of the Gaussian. Moreover as $V$ is even, so is $\mu_t$.  
\end{lem}

\begin{proof}
By differentiation,
\[
-F'(t)=-\frac{Z'_{t}}{Z_{t}}=\frac{1}{Z_{t}}\,\E_{\gamma_{n}}[\de_{t}V(X_{t})\cdot e^{-V(X_{t})}]=\E_{\mu_{t}}\de_{t}V(X_{t})\,.
\]
Using $X_{t}=tX$ and $\D(\log \det A(t)) = \Tr (A(t)^{-1}  \,\D A(t))$, a simple calculation gives
\[
\de_{t}V(X_{t})=\frac{2}{t}\,\bbrack{b\Tr X^{2}_{t}+\alpha\Tr\bpar{X^{4}_{t}\,(I-X^{2}_{t})^{-1}}}\,.
\]
The expectation $\E_{\mu_t}\de_{t}V(X_{t})$ can be simplified further. As $(I-X^{2}_{t})^{-1}=\frac{1}{2}\,[(I-X_{t})^{-1}+(I+X_{t})^{-1}]$, and as $\mu_{t}$ is even, $(I-X_t)^{-1}$ and $(I+X_t)^{-1}$ have the same distribution. Hence, one has
\[
\E_{\mu_t}\Tr\bpar{X^4_{t}\,(I-X^{2}_t)^{-1}} = 
\E_{\mu_t} \Tr\bpar{X^{4}_{t}\,(I-X_{t})^{-1}}\,,\]
which gives the claimed expression for $-F'(t)$.
\end{proof}

\paragraph{Moment Bounds.}
Given Lemma \ref{lem:F'(t)}, it remains to bound the terms in $-F'(t)$. 
Let us define $R_{t}:=(I-X_{t})^{-1}$, and let $S_t=\E_{\mu_t}R_t$. Also recall that $\eta = n/r^2$. We will show the following bounds.

\begin{lem}[Moment bounds.]\label{lem:moments}
 For $0\le t\le1$, we have:\vspace{-8pt}
 \begin{enumerate}
     \item (Second moment bound.) $\E_{\mu_t} X_t^2 \preceq S_t-I \preceq 8 \eta t^2 I$, and hence  $\E_{\mu_t} \Tr(X_t^2) \leq 8 \eta t^2 d$.  
     \item (Fourth moment bound.) $\E_{\mu_t}\Tr(X_t^4R_t) \le 128 \eta^2  t^4 d$.
 \end{enumerate}

\end{lem}
\paragraph{Finishing the Proof.} This directly gives Theorem \ref{thm:spencer} as follows. 
Indeed by Lemma \ref{lem:F'(t)} and \ref{lem:moments},
\[-F'(t) \leq O(b\eta t d  + \alpha \eta^2 t^3 d) = O(\eta^{3/2}d) =O(n/L^3),\]
where the first equality uses our choice of $b,\alpha$ in \eqref{eq:general-parameters}\footnote{This asymmetric dependence on $\eta$ in the moment bounds is precisely the reason for choosing $b$ and $\alpha$ in this way.  } and $t\in[0,1]$, and the last equality uses the value of $\eta$ from \eqref{eq:general-parameters}.

As $F(0)=0$, choosing $L$ large enough, gives $F(1) \geq -n/10$, implying \eqref{eq3} and hence Theorem \ref{thm:spencer}.

\subsection{Proving Moment Bounds via Matrix Brascamp--Lieb}
Fix some $t \in [0, 1]$. For ease of notation, we suppress $t$ from $X_{t},R_{t},S_t$ and $\mu_{t}$ (unless required). Recall that $R_t =(I-X_t)^{-1}$ and $S_t = \E_{\mu_t} R_t$.
We first state a simple fact.
\begin{lem}
\label{lem:identity} 
$XRX=R-I-X$ and thus $\E_{\mu}[XRX] = S-I$. Moreover, $\E_\mu X^2 \preceq \E_\mu XRX$.
\end{lem}
\begin{proof}
As $R\,(I-X)=I$ by definition, we have 
$RX=XR=R-I$. Using this again gives  $X^2R= X (R-I) = R-I-X$.
Taking expectation and using $\E_{\mu}X=0$, gives $S-I=\E_{\mu}[XRX]$.

For the second part, using $R = I+X+ X^2 +\cdots $, multiplying by $X$ on both sides, and taking expectations gives $\E_\mu[XRX]=\E_{\mu}[\sum_{k\geq 1}X^{2k}]\succeq \E_{\mu}[X^{2}]$.
\end{proof}

\paragraph{Matrix Brascamp--Lieb Inequality.}
We now prove Lemma \ref{lem:moments}.
By Lemma \ref{lem:identity}, we need to bound $\E_\mu \Tr (X^2R)$ and $\E_\mu \Tr (X^4 R)$.
The key tool will be the matrix Brascamp--Lieb (BL) inequality~\cite{CorderoErausquin18,AkbasSra26}. 
It says the following. 

Let
$M$ be the matrix density defined on $\R^n$ by $M(g)\,\D g=R(g)\,\D\mu(g)$, 
where $R(g)$ is a $d\times d$ PSD matrix, and $\mu(g)$ is a (scalar) probability density.
For a vector-valued function $f:\R^n \rightarrow \R^d$, denote the weighted mean $\bar{f}_{M}=(\E_{\mu}R)^{-1}\E_{\mu}[Rf]$. Then,
\begin{equation}
~\eqref{eq:matrix-poincare}:\,
\var_{M}f=\E_{\mu}[f^{\T}Rf]-\bar{f}^{\T}_{M}(\E_{\mu}R)\,\bar{f}_{M}\le \frac{1}{\kappa_M}\, \E_{\mu}\Bbrack{\sum^{n}_{i=1}(\de_{i}f)^{\T}R\,(\de_{i}f)}\,.\label{eq:general-matrixBL}
\end{equation}
where $\kappa_M$ is a curvature lower bound for $M$ (see \S\ref{app:matrix-poincare}).

For our purposes with $\D M = R\,\D \mu$, we can use this BL inequality with $\kappa_M=1/2$ . We will verify this curvature condition in \S\ref{app:curvature}; morally, this is because even though $R = (I-X)^{-1}$ is not log-concave, $\mu$ is sufficiently strongly-log-concave (due to Gaussianity and convexity of $V$).

\subsubsection{Bounding Second Moment}
We show that $\E_\mu[XRX] \preceq 8t^2\eta I$ by bounding the largest eigenvalue $\lda_t$ of $S_t$.

Let $z =z_t \in \R^d$ be a unit eigenvector of $S_t$ corresponding to $\lambda_t$. 
We use the matrix BL \eqref{eq:general-matrixBL} to $f=Xz$.
Let us compute each term in BL. 

For the first term, as $\E_{\mu}[XRX]=S-I$, we have $\E_{\mu}[f^{\T}Rf]=z^{\T}(S-I)\,z=\lda_{t}-1$.

For the mean, using $RX=R-I$, we have $\E_{\mu}[Rf]=(\lda_{t}-1)\,z$
and $\bar{f}_{M}=(\lda_{t}-1)\,S^{-1}z$. Hence,
\[
\bar{f}^{\T}_{M}S\bar{f}_{M}=(\lda_{t}-1)^{2}z^{\T}S^{-1}z=(\lda_{t}-1)^{2}/\lda_{t}\,.
\]

For the RHS of BL, as $f = Xz = \frac{t}{r}\, (\sum_i A_ig_i)\,z$, we have $\de_{i}f=\frac{t}{r}\,A_{i}z$, and
as $S \preceq \lda_{t}I$ and $A_i^2 \preceq I$,
\[
\E_{\mu}\Bbrack{\sum^{n}_{i=1}(\de_{i}f)^{\T}R\,(\de_{i}f)}=\frac{t^{2}}{r^{2}} \sum_{i}z^{\T}A_{i}SA_{i}z\le\frac{\lda_{t}nt^{2}}{r^{2}} = \lambda_t t^2 \eta\,.
\]
Therefore, BL gives $\lda_{t}-1-(\lda_{t}-1)^{2}/\lda_{t}\le 2 \eta t^2 \lda_{t}$, which upon rearranging gives
\begin{equation}
\lda_{t}-1\le2\eta t^2 \lda^{2}_{t}\,.\label{eq:general-lambda}
\end{equation}
As $L\ge8$, we have $\eta t^2\le\eta\le\nicefrac{1}{64}$, and
solving the inequality gives
$\lda_{t}\le\frac{2}{1+\sqrt{1-8\eta t^2}}\le 1+8\eta t^2$.
Together with $S\succeq I$, this implies
$0\preceq S_{t}-I\preceq 8t^{2}\eta I\,.$
\subsubsection{Bounding Fourth Moment}
We now handle $\E_{\mu} \Tr[X^{4}R]  =\Tr(\E_{\mu}[X^{4}R])$.
Choose an orthonormal eigenbasis $z_{1},\ldots,z_{d}$ of $S$, with
eigenvalues $\lda_{1},\ldots,\lda_{d}$. For each $a\in[d]$, we set
$f_{a}=X^{2}z_{a}$ so that 
\[
\E_{\mu}[f^{\T}_{a}Rf_{a}]=z^{\T}_{a}\E_{\mu}[X^{2}RX^{2}]\,z_{a}\,.
\]
Since $\Tr(\E_{\mu}[X^{2}RX^{2}])=\sum_{a}\E_{\mu}[f^{\T}_{a}Rf_{a}]$,
we will apply~\eqref{eq:general-matrixBL} to each $f_{a}$ and sum
over $a$.

We now compute the second term in the BL inequality for $f_{a}=X^{2}z_{a}$. 

Using the identity $RX^{2}=R-I-X$ (Lemma~\ref{lem:identity}) and $\E_{\mu}X=0$ gives
\[
\E_{\mu}[Rf_{a}]= \E_\mu [RX^2 z_a]=(\lda_{a}-1)\,z_{a}\]
and thus $\bar{f}_{a,M}=(\lda_{a}-1)\,S^{-1}z_{a}$. So the second term 
\[\bar{f}^{\T}_{a,M}S\bar{f}_{a,M}=(\lda_{a}-1)^{2}\,z^{\T}_{a}S^{-1}z_{a}=(\lda_{a}-1)^{2}/\lda_{a} \leq 4 \eta t^2 \, (\lambda_a-1)\,,\]
where the last step uses that
$(\lambda_a -1)/\lambda_a \leq (\lambda_t -1)/\lambda_t$ which by \eqref{eq:general-lambda} is at most $4 \eta t^2 $ (as $\lambda_t\leq 2$).
Summing over $a$ and using $1\le\lda_{a}\le\lda_t$ gives 
\[
\sum_{a}\bar{f}^{\T}_{a,M}S\bar{f}_{a,M} \le 4\eta t^2 \sum_{a}(\lda_{a}-1)=4\eta t^2 \Tr(S-I)\,.
\]

\paragraph{RHS of the BL Inequality.}
For $f_a = X^2 z_a$, the derivative is $\de_{i}f_{a}=\frac{t}{r}\,(A_{i}X+XA_{i})\,z_{a}$.

Using $(u+w)^{\T}R(u+w)\le2u^{\T}Ru+2w^{\T}Rw$, we have
\[
\sum_{a}\E_{\mu}\Bbrack{\sum_{i}(\de_{i}f_{a})^{\T}R\,(\de_{i}f_{a})}\le\frac{2t^{2}}{r^{2}}\,\underbrace{\E_{\mu}\Bbrack{\sum_{i}\Tr(XA_{i}RA_{i}X)}}_{T_{1}}+\frac{2t^{2}}{r^{2}}\,\underbrace{\E_{\mu}\Bbrack{\sum_{i}\Tr(A_{i}XRXA_{i})}}_{T_{2}}\,.
\]
As $\sum_{i}A^{2}_{i}\preceq nI$, the second term satisfies 
$T_{2}=\E_{\mu}[\Tr(XRX\sum_{i}A^{2}_{i})]\leq n\,\E_{\mu}\Tr(XRX)=n\Tr(S-I)\,.$

For $T_{1}$, we write $T_1 = \E_{\mu}\Tr (BR)$ for  $B:=\sum_{i}A_{i}X^{2}A_{i}$. Since $X^{2}\preceq I$
and $\norm{A_{i}}\leq1$, we have 
$0\preceq B\preceq\sum_{i}A^{2}_{i}\preceq nI$ and $\E_\mu \Tr B \le n\E_\mu \Tr(X^{2}) \leq n\,\E_\mu \Tr(XRX)$.

Now note that $BR = B\,(I+X+XRX)$. The matrix $B$ is even in $g$, so $\E_{\mu}\Tr(BX)=0$.
Also as $B \preceq nI$, we have  $\Tr(BXRX)\leq n\Tr(XRX)$. Thus,
\[
T_{1}=\E_{\mu}\Tr(BR)=\E_{\mu}\Tr B+\E_{\mu}\Tr(BX^{2}R)\le n\,\E_{\mu}\Tr(XRX)+n\,\E_\mu \Tr(XRX)\,.
\]
Hence, $T_{1}\le2n\Tr(S-I)$ due to $\E_\mu[XRX] = S- I$. 
Combining these, the matrix BL implies
\begin{align*}
\Tr(\E_{\mu}[X^{2}RX^{2}]) &\le\sum_{a}\bar{f}^{\T}_{a,M}S\bar{f}_{a,M}+2\sum_{a}\E_{\mu}\Bbrack{\sum_{i}(\de_{i}f_{a})^{\T}R\,(\de_{i}f_{a})} \\
& \le16\eta t^2 \Tr(S-I)\le128d\eta^2 t^4 \,.\label{eq:general-quartic}
\end{align*}

\appendix

\section{Matrix Brascamp--Lieb Inequality\label{app:matrix-poincare}}

In this section, we first recall the standard Brascamp--Lieb inequality and then its natural matrix counterpart introduced by Cordero-Erausquin~\cite[Theorem~4]{CorderoErausquin18}.
Throughout this appendix, $\D g$ denotes Lebesgue measure.

\subsection{Weighted Means, Variance, and Curvature}

For a positive integrable scalar density $w=e^{-\phi}$ on $\Rn$,
write $\bar{h}_{w}=(\int w)^{-1}\int wh$. The standard Brascamp--Lieb
inequality reads that when $\hess\phi\succ0$,
\[
\var_{w}h:= \int(h - \bar{h}_w)^2\,w\,\D g \le\int\inner{\bpar{\hess\phi}^{-1}\grad h,\grad h}\,w\,\D g\,,
\]
In particular, $\hess\phi\succeq\kappa I_{n}\succ0$ implies the Poincar\'e
inequality; $\var_{w}h\leq\kappa^{-1}\int\norm{\grad h}^{2}\,w\,\D g$.

For the matrix extension, let $\Omega\subseteq\Rn$ be nonempty, open,
and convex, and let $M:\Omega\to\pd$ be a smooth matrix weight with
$\int_{\Omega}\norm M\,\D g<\infty$. The matrix BL is obtained by
replacing the scalar density $w$ by the matrix density $M$. For
$f:\Omega\to\Rd$, the weighted mean is 
\[
\bar{f}_{M}=\mc Z^{-1}_{M}\int_{\Omega}Mf\,\D g\qquad\text{for }\quad \mc Z_{M}:=\int_{\Omega}M\,\D g\,.
\]
 The weighted variance is
\[
\var_{M}f=\int_{\Omega}(f-\bar{f}_{M})^{\T}M\,(f-\bar{f}_{M})\,\D g=\int_{\Omega}f^{\T}Mf\,\D g-\bar{f}^{\T}_{M}\mc Z_{M}\bar{f}_{M}\,.
\]
We now need the `` $\hess(-\log M)$''. To this end, Cordero-Erausquin
defined the curvature operator $\Theta^{M}:(\Rd)^{n}\to(\Rd)^{n}$
as follows: for $U=(u_{1},\dots,u_{n})\in(\Rd)^{n}$
\[
(\Theta^{M}U)_{i}=\sum^{n}_{j=1}\de_{i}(M^{-1}\de_{j}M)\,u_{j}=-\sum^{n}_{j=1}M^{-1}\,\bpar{\underbrace{(\de_{i}M)\,M^{-1}(\de_{j}M)-\de_{ij}M}_{=:(\mc K_{M})_{ij}}}\,u_{j}\,.
\]
Note that for a scalar weight $M=e^{-\phi}$, this reduces to $-\Theta^{M}=\hess\phi$. 

On block vectors $U=(u_{1},\ldots,u_{n})$ and $W=(w_{1},\ldots,w_{n})$
in $(\Rd)^{n}$, we define
\[
G=I_{n}\otimes M\qquad\text{and}\qquad\inner{U,W}_{G}=\sum^{n}_{i=1}u^{\T}_{i}Mw_{i}\,.
\]
Let us deduce the symmetric block-matrix representation of $\Theta^{M}$
with respect to the metric $G$:
\[
\inner{U,(-\Theta^{M})\,U}_{G}=\sum^{n}_{i=1}\sum^{n}_{j=1}u^{\T}_{i}\bpar{(\de_{i}M)\,M^{-1}(\de_{j}M)-\de_{ij}M}\,u_{j}\,.
\]
Hence, w.r.t.\ $G=I_{n}\otimes M$, the symmetric block-matrix representation
of $-\Theta^{M}$ is 
\[
\mc K_{M}:=-(I_{n}\otimes M)\,\Theta^{M}=[(\de_{i}M)\,M^{-1}(\de_{j}M)-\de_{ij}M]^{n}_{i,j=1}\,.
\]
All block matrices here have $n\times n$ blocks of size $d\times d$.

When $\Omega=\Rn$ and $-\Theta^{M}$ is positive definite in the
$G$-metric\footnote{$A$ is called PD in the $G$-metric if $A^\T G=GA$ and $\inner{u,Au}_{G}=u^{\T}(GA)u>0$ for $u\neq 0$.},
the matrix BL inequality gives
\[
\var_{M}f\le\int_{\Rn}\inner{(-\Theta^{M})^{-1}\grad f,\grad f}_{G}\,\D g\,.
\]
We use the following convex-domain version proven in \cite[Lemma~2.3]{AkbasSra26}.
\begin{lem}[Matrix-weighted Poincaré inequality]
 \label{lem:matrix-poincare} Suppose that $\mc K_{M}\succeq\kappa G$
on convex $\Omega$ for some $\kappa>0$. Then, every $C^{1}$ function
$f:\Omega\to\Rd$ with $\int_{\Omega}f^{\T}Mf\,\D g<\infty$ satisfies
\begin{equation}
\var_{M}f\le\frac{1}{\kappa}\int_{\Omega}\norm{\nabla f}^{2}_{G}\,\D g=\frac{1}{\kappa}\int\sum^{n}_{i=1}(\de_{i}f)^{\T}M\,(\de_{i}f)\,\D g\,.\label{eq:matrix-poincare}
\end{equation}
\end{lem}

The proof in \cite{CorderoErausquin18} relies on the weighted Bochner
argument, and in such proof, a convex domain only adds a ``favorable''
term in satisfying target conditions, so the convex version of the
matrix BL is plausible, and \cite{AkbasSra26} essentially confirms
this.

\section{Curvature of the Barrier\label{app:curvature}}

We prove $\mc K_{M}\succeq\frac{1}{2}\,(I_{n}\otimes M)$ (in Lemma~\ref{lem:matrix-poincare})
of the barrier \eqref{eq:general-potential}. 

\paragraph{Preliminary computations.}

Fix $t$, omit its subscript, and write
\[
D_{i}:=\de_{i}X=\frac{t}{r}\,A_{i}\,,\qquad R=(I-X)^{-1}\,,\qquad M=\rho R\,,
\]
where $-\log\rho=\nicefrac{\norm g^{2}}{2}+V(X)+O(1)$. For $i,j\in [n]$, we set
\[
H_{ij}:=\de_{i}\de_{j}V\bpar{X(g)}\qquad\text{and}\qquad E_{i}:=R^{1/2}D_{i}R^{1/2}\in \Rdd\,.
\]

Using $M^{-1}\de_{j}M=R^{-1}\de_{j}R+(\de_{j}\log\rho)\,I$, direct
differentiation gives
\begin{equation}
\mc K_{M}=\rho\mc K_{R}+\bpar{\hess(-\log\rho)}\otimes M=\rho\mc K_{R}+(I_{n}+H)\otimes M\,.\label{eq:ordinary-split}
\end{equation}
Standard differentiation leads to $\de_{i}R  =RD_{i}R$, so
\begin{equation}\label{eq:resolvent-curvature}
    (\mc K_R)_{ij} = (\de_{i}R)R^{-1}\,(\de_{j}R) - \de_{ij}R
= RD_{i}RD_{j}R - RD_{i}RD_{j}R - RD_{j}RD_{i}R = -RD_{j}RD_{i}R\,.
\end{equation}
Combining~\eqref{eq:ordinary-split} and~\eqref{eq:resolvent-curvature},
we divide by $\rho$ and conjugate each block by $R^{-1/2}$: 
\[
\frac{1}{\rho}\,(I_{n}\otimes R^{-1/2})\,\mc K_{M}\,(I_{n}\otimes R^{-1/2})=I_{nd}+H\otimes I_{d}-[E_{j}E_{i}]_{i,j}\,.
\]
Since $I_{n}\otimes M$ becomes $I_{nd}$ after the conjugate, $\mc K_{M}\succeq \frac12 \,(I_n \otimes M)$ is equivalent to 
\begin{equation}
[E_{j}E_{i}]_{i,j}\preceq H\otimes I_{d}+\frac{1}{2}\,I_{nd}\,.\label{eq:modified-normalized}
\end{equation}

Recall $V(X)=b\Tr(X^{2})+\alpha\Tr s(X)$ for $s(x):=-\log(1-x^{2})-x^{2}$, and write $E_{i}=D_{i}+\Delta_{i}$, where $\Delta_{i}:=R^{1/2}D_{i}R^{1/2}-D_{i}$.
We then bound the three terms in $[E_{j}E_{i}]_{i,j}=[D_{j}D_{i}]_{i,j}+[\Delta_{j}\Delta_{i}]_{i,j}+[D_{j}\Delta_{i}+\Delta_{j}D_{i}]_{i,j}$ as follows: 
\begin{align*}
T_{1}:=[D_{j}D_{i}]_{i,j} & \preceq \frac{nt^{2}}{r^{2}}\,I_{nd}\preceq \frac{1}{64}\,I_{nd}\,,\\
T_{2}:=[\Delta_{j}\Delta_{i}]_{i,j} & \preceq[\hess_{g}\Tr s(X)]\otimes I_{d}\,,\\
T_{3}:=[D_{j}\Delta_{i}+\Delta_{j}D_{i}]_{i,j} & \preceq\bpar{b\,[\Tr(D_{i}D_{j})]_{i,j}+\frac{1}{b}\,\hess_{g}\Tr s(X)}\otimes I_{d}\,.
\end{align*}
Since $\hess\Tr(X^{2}) = 2\,[\Tr(D_{i}D_{j})]_{i,j}$ by direct computation, the RHS of~\eqref{eq:modified-normalized} is
\[
H\otimes I_{d}+\frac{1}{2}\,I_{nd}=\bpar{2b\,[\Tr(D_{i}D_{j})]_{i,j}+\alpha\hess_{g}\Tr s(X)}\otimes I_{d}+\frac{1}{2}\,I_{nd}\,.
\]
Hence, we couple (i) $\frac{nt^{2}}{r^{2}}\,I_{nd}$ with $\half\,I_{nd}$, (ii)
$(1+b^{-1})\,[\hess_{g}\Tr s(X)]\otimes I_{d}$ with $\alpha\,[\hess_{g}\Tr s(X)]\otimes I_{d}$
(since $1+b^{-1}\leq\alpha$), and (iii) $b\,[\Tr(D_{i}D_{j})]_{i,j}\otimes I_{d}$
with $2b\,[\Tr(D_{i}D_{j})]_{i,j}\otimes I_{d}$.

\paragraph{Bounding $T_{1}$.}

Since $D_{i}=\frac{t}{r}\,A_{i}$ and $\norm{A_{i}}\le1$, we have
$\norm{\sum_{i}D^{2}_{i}}\leq\frac{nt^{2}}{r^{2}}$. For
arbitrary $z_{1},\ldots,z_{n}\in\Rd$, Cauchy--Schwarz yields 
\[
\sum_{i,j}z^{\T}_{i}D_{j}D_{i}z_{j}=\sum_{i,j}\inner{D_{j}z_{i},D_{i}z_{j}}\le\sum_{i,j}\norm{D_{j}z_{i}}^{2}\le\veps_{t}\sum_{i}\norm{z_{i}}^{2}\,.
\]

\paragraph{Bounding $T_{2}$.}

For $T=[T_{ij}]^{n}_{i,j=1}$ with $d\times d$ blocks, \cite[Proposition~2.1]{Lin16} states that if $[T_{ij}]_{i,j}\succeq0$, then $[\Tr(T_{ji})\,I_{d}-T_{ji}]_{i,j}\succeq0$\footnote{This is a block generalization of $T\preceq\Tr T\cdot I_{d}$ when
$T\succeq0$.}.
 Applying this to $T_{ij}=\Delta_{i}\Delta_{j}$ yields
$[\Delta_{j}\Delta_{i}]_{i,j}\preceq[\Tr(\Delta_{i}\Delta_{j})]_{i,j}\otimes I_{d}$.
Thus, it suffices to show
\begin{equation}
[\Tr(\Delta_{i}\Delta_{j})]_{i,j}\preceq \hess_{g}\Tr s(X)\,.\label{eq:modified-hessian-direction}
\end{equation}

Let us compute the Hessian of $\Tr s(X)$. Define $V_{-}=-\log\det(I-X)$
and $V_{+}=-\log\det(I+X)$. Then, $\Tr s(X)=V_{-}+V_{+}-\Tr X^{2}$.
Direct differentiation yields $\de_{j}V_{-} =\Tr(RD_{j})$ and $\de_{ij}V_{-} =\Tr(RD_{i}RD_{j})=\Tr(E_{i}E_{j})$, and the similar calculation for $V_{+}$ leads to
\[
\de_{ij}\Tr s(X)=\Tr(RD_{i}RD_{j})+\Tr\bpar{(I+X)^{-1}D_{i}\,(I+X)^{-1}D_{j}}-2\Tr(D_{i}D_{j})\,.
\]

Diagonalize $X$, with eigenvalues $x_a$, and set $r_{a}=(1-x_{a})^{-1}$ and $p_{a}=(1+x_{a})^{-1}$. In this basis, $R=\diag(r_{1},\ldots,r_{d})$ and $(I+X)^{-1}=\diag(p_{1},\ldots,p_{d})$, so
\[
\de_{ij}\Tr s(X)=\sum_{a,b}(r_{a}r_{b}+p_{a}p_{b}-2)\,(D_{i})_{ab}(D_{j})_{ab}\,.
\]
For $\Tr(\Delta_{i}\Delta_{j})$, we have $(\Delta_{i})_{ab}=(\sqrt{r_{a}r_{b}}-1)(D_{i})_{ab}$. 
Since $\Delta_{i}$ is symmetric,
\[
\Tr(\Delta_{i}\Delta_{j})=\sum_{a,b}(\Delta_{i})_{ab}(\Delta_{j})_{ab}=\sum_{a,b}(\sqrt{r_{a}r_{b}}-1)^{2}\,(D_{i})_{ab}(D_{j})_{ab}\,.
\]

We now compare quadratic forms. For $u\in\Rn$, using $\sum_{i,j}u_{i}u_{j}\,(D_{i})_{ab}(D_{j})_{ab}=(\sum_{i}u_{i}\,(D_{i})_{ab})^{2}$,
\begin{align*}
u^{\T}[\hess_{g}\Tr s(X)]\,u & =\sum_{a,b}(r_{a}r_{b}+p_{a}p_{b}-2)\,\Bpar{\sum_{i}u_{i}\,(D_{i})_{ab}}^{2}\,,\\
u^{\T}[\Tr(\Delta_{i}\Delta_{j})]_{i,j}\,u & =\sum_{a,b}(\sqrt{r_{a}r_{b}}-1)^{2}\,\Bpar{\sum_{i}u_{i}\,(D_{i})_{ab}}^{2}\,.
\end{align*}
Hence, it suffices to show that $r_{a}r_{b}+p_{a}p_{b}-2\ge(\sqrt{r_{a}r_{b}}-1)^{2}$ (equivalently, $p_{a}p_{b}+2\sqrt{r_{a}r_{b}}\geq 3$).
Apply AM-GM to the three positive numbers $p_{a}p_{b}$,
$\sqrt{r_{a}r_{b}}$, and $\sqrt{r_{a}r_{b}}$: 
\[
p_{a}p_{b}+2\sqrt{r_{a}r_{b}}\ge3\,(p_{a}p_{b}r_{a}r_{b})^{1/3}\underset{(i)}{\geq}3\,,
\]
where $(i)$ follows from $\abs{x_{a}},\abs{x_{b}}<1$ and $p_{a}p_{b}r_{a}r_{b}=\frac{1}{(1-x^{2}_{a})(1-x^{2}_{b})}\ge1$.

\paragraph{Bounding $T_{3}$.}
By the simple block Cauchy--Schwarz~\cite[Lemma 3.5]{AkbasSra26}, for block vectors $z_i \in \Rd$,
\[
\sum_{i,j}z^{\T}_{i}(D_{j}\Delta_{i}+\Delta_{j}D_{i})\,z_{j}
= 2\sum_{i,j}z_i^\T D_j \Delta_i z_j \leq 2\,
\Bpar{\sum_{i,j}\Tr(D_i D_j)\,z_i^\T z_j}^{1/2}
\Bpar{\sum_{i,j}\Tr(\Delta_i \Delta_j)\,z_i^\T z_j}^{1/2}\,.
\]
Using Young's inequality (i.e., $2\,(xy)^{1/2}\leq cx + c^{-1}y$) and \eqref{eq:modified-hessian-direction},
\[
[D_{j}\Delta_{i}+\Delta_{j}D_{i}]_{i,j}\preceq b\,[\Tr(D_{i}D_{j})]_{i,j}\otimes I_{d}+\frac{1}{b}\,[\Tr(\Delta_{i}\Delta_{j})]_{i,j}\otimes I_{d}
\preceq b\,[\Tr(D_{i}D_{j})]_{i,j}\otimes I_{d}+\frac{1}{b}\,\hess_{g}\Tr s(X)\otimes I_{d}\,.
\]
\bibliographystyle{alpha}
\bibliography{main}
\end{document}

%% file: math-macros.tex
\input{_macros.tex}
\global\long\def\on#1{\operatorname{#1}}%

\global\long\def\bw{\mathsf{Ball\ walk}}%
\global\long\def\sw{\mathsf{Speedy\ walk}}%
\global\long\def\gw{\mathsf{Gaussian\ walk}}%
\global\long\def\ps{\mathsf{Proximal\ sampler}}%
\global\long\def\dw{\mathsf{Dikin\ walk}}%

\global\long\def\chr{\mathsf{Coordinate\ Hit\text{-}and\text{-}Run}}%
\global\long\def\har{\mathsf{Hit\text{-}and\text{-}Run}}%
\global\long\def\gc{\mathsf{Gaussian\ cooling}}%
\global\long\def\ino{\mathsf{\mathsf{In\text{-}and\text{-}Out}}}%
\global\long\def\tgc{\mathsf{Tilted\ Gaussian\ cooling}}%
\global\long\def\PS{\mathsf{PS}}%
\global\long\def\psunif{\mathsf{PS}_{\textup{unif}}}%
\global\long\def\psexp{\mathsf{PS}_{\textup{exp}}}%
\global\long\def\psann{\mathsf{PS}_{\textup{ann}}}%
\global\long\def\psgauss{\mathsf{PS}_{\textup{Gauss}}}%
\global\long\def\eval{\mathsf{Eval}}%
\global\long\def\mem{\mathsf{Mem}}%

\global\long\def\O{O}%
\global\long\def\Otilde{\widetilde{O}}%
\global\long\def\Omtilde{\widetilde{\Omega}}%

\global\long\def\E{\mathbb{E}}%
\global\long\def\Z{\mathbb{Z}}%
\global\long\def\P{\mathbb{P}}%
\global\long\def\N{\mathbb{N}}%

\global\long\def\R{\mathbb{R}}%
\global\long\def\Rd{\mathbb{R}^{d}}%
\global\long\def\Rdd{\mathbb{R}^{d\times d}}%
\global\long\def\Rn{\mathbb{R}^{n}}%
\global\long\def\Rnn{\mathbb{R}^{n\times n}}%

\global\long\def\psd{\mathbb{S}^{d}_{+}}%
\global\long\def\pd{\mathbb{S}^{d}_{++}}%

\global\long\def\defeq{\stackrel{\mathrm{{\scriptscriptstyle def}}}{=}}%

\global\long\def\veps{\varepsilon}%
\global\long\def\lda{\lambda}%
\global\long\def\vphi{\varphi}%
\global\long\def\K{\mathcal{K}}%

\global\long\def\half{\frac{1}{2}}%
\global\long\def\nhalf{\nicefrac{1}{2}}%
\global\long\def\texthalf{{\textstyle \frac{1}{2}}}%
\global\long\def\ltwo{L^{2}}%

\global\long\def\ind{\mathds{1}}%
\global\long\def\op{\mathsf{op}}%
\global\long\def\ch{\mathsf{Ch}}%
\global\long\def\kls{\mathsf{KLS}}%
\global\long\def\ts{\mathsf{Ts}}%
\global\long\def\hs{\textup{HS}}%
\global\long\def\ls{\textup{LS}}%

\global\long\def\cpi{C_{\mathsf{PI}}}%
\global\long\def\cipi{C_{\mathsf{IPI}}}%
\global\long\def\clsi{C_{\mathsf{LSI}}}%
\global\long\def\cch{C_{\mathsf{Ch}}}%
\global\long\def\clch{C_{\mathsf{logCh}}}%
\global\long\def\cexp{C_{\mathsf{exp}}}%
\global\long\def\cgauss{C_{\mathsf{Gauss}}}%

\global\long\def\chooses#1#2{_{#1}C_{#2}}%

\global\long\def\frob{\on F}%

\global\long\def\vol{\on{vol}}%

\global\long\def\sym{\on{sym}}%

\global\long\def\law{\on{law}}%

\global\long\def\tr{\on{tr}}%

\global\long\def\diag{\on{diag}}%

\global\long\def\diam{\on{diam}}%

\global\long\def\poly{\on{poly}}%

\global\long\def\polylog{\on{polylog}}%

\global\long\def\Diag{\on{Diag}}%

\global\long\def\inter{\on{int}}%

\global\long\def\esssup{\on{ess\,sup}}%

\global\long\def\proj{\on{Proj}}%

\global\long\def\e{\mathrm{e}}%

\global\long\def\id{\mathrm{id}}%

\global\long\def\supp{\on{supp}}%

\global\long\def\spanning{\on{span}}%

\global\long\def\rows{\on{row}}%

\global\long\def\cols{\on{col}}%

\global\long\def\rank{\on{rank}}%

\global\long\def\T{\mathsf{T}}%

\global\long\def\bs#1{\boldsymbol{#1}}%

\global\long\def\eu#1{\EuScript{#1}}%

\global\long\def\mb#1{\mathbf{#1}}%

\global\long\def\mbb#1{\mathbb{#1}}%

\global\long\def\mc#1{\mathcal{#1}}%

\global\long\def\mf#1{\mathfrak{#1}}%

\global\long\def\ms#1{\mathscr{#1}}%

\global\long\def\mss#1{\mathsf{#1}}%

\global\long\def\msf#1{\mathsf{#1}}%

\global\long\def\cQ{\msf Q}%
\global\long\def\cV{\msf V}%
\global\long\def\cC{\msf C}%

\global\long\def\textint{{\textstyle \int}}%
\global\long\def\Dd{\mathrm{D}}%
\global\long\def\D{\mathrm{d}}%
\global\long\def\grad{\nabla}%
 
\global\long\def\hess{\nabla^{2}}%
 
\global\long\def\lapl{\triangle}%
 
\global\long\def\deriv#1#2{\frac{\D#1}{\D#2}}%
 
\global\long\def\pderiv#1#2{\frac{\partial#1}{\partial#2}}%
 
\global\long\def\de{\partial}%
\global\long\def\lagrange{\mathcal{L}}%
\global\long\def\Div{\on{div}}%

\global\long\def\Gsn{\mathcal{N}}%
 
\global\long\def\BeP{\textnormal{BeP}}%
 
\global\long\def\Ber{\textnormal{Ber}}%
 
\global\long\def\Bern{\textnormal{Bern}}%
 
\global\long\def\Bet{\textnormal{Beta}}%
 
\global\long\def\Beta{\textnormal{Beta}}%
 
\global\long\def\Bin{\textnormal{Bin}}%
 
\global\long\def\BP{\textnormal{BP}}%
 
\global\long\def\Dir{\textnormal{Dir}}%
 
\global\long\def\DP{\textnormal{DP}}%
 
\global\long\def\Exp{\textnormal{Exp}}%
 
\global\long\def\Gam{\textnormal{Gamma}}%
 
\global\long\def\GEM{\textnormal{GEM}}%
 
\global\long\def\HypGeo{\textnormal{HypGeo}}%
 
\global\long\def\Mult{\textnormal{Mult}}%
 
\global\long\def\NegMult{\textnormal{NegMult}}%
 
\global\long\def\Poi{\textnormal{Poi}}%
 
\global\long\def\Pois{\textnormal{Pois}}%
 
\global\long\def\Unif{\textnormal{Unif}}%

\global\long\def\bpar#1{\bigl(#1\bigr)}%
\global\long\def\Bpar#1{\Bigl(#1\Bigr)}%

\global\long\def\abs#1{|#1|}%
\global\long\def\babs#1{\bigl|#1\bigr|}%
\global\long\def\Babs#1{\Bigl|#1\Bigr|}%

\global\long\def\snorm#1{\|#1\|}%
\global\long\def\bnorm#1{\bigl\Vert#1\bigr\Vert}%
\global\long\def\Bnorm#1{\Bigl\Vert#1\Bigr\Vert}%

\global\long\def\sbrack#1{[#1]}%
\global\long\def\bbrack#1{\bigl[#1\bigr]}%
\global\long\def\Bbrack#1{\Bigl[#1\Bigr]}%

\global\long\def\sbrace#1{\{#1\}}%
\global\long\def\bbrace#1{\bigl\{#1\bigr\}}%
\global\long\def\Bbrace#1{\Bigl\{#1\Bigr\}}%

\global\long\def\Abs#1{\left\lvert #1\right\rvert }%
\global\long\def\Par#1{\left(#1\right)}%
\global\long\def\Brack#1{\left[#1\right]}%
\global\long\def\Brace#1{\left\{  #1\right\}  }%

\global\long\def\inner#1{\langle#1\rangle}%
 
\global\long\def\binner#1#2{\left\langle {#1},{#2}\right\rangle }%

\global\long\def\norm#1{\lVert#1\rVert}%
\global\long\def\onenorm#1{\norm{#1}_{1}}%
\global\long\def\twonorm#1{\norm{#1}_{2}}%
\global\long\def\infnorm#1{\norm{#1}_{\infty}}%
\global\long\def\fronorm#1{\norm{#1}_{\text{F}}}%
\global\long\def\nucnorm#1{\norm{#1}_{*}}%
\global\long\def\staticnorm#1{\|#1\|}%
\global\long\def\statictwonorm#1{\staticnorm{#1}_{2}}%

\global\long\def\mmid{\mathbin{\|}}%

\global\long\def\otilde#1{\widetilde{O}(#1)}%
\global\long\def\wtilde{\widetilde{W}}%
\global\long\def\wt#1{\widetilde{#1}}%

\global\long\def\KL{\msf{KL}}%
\global\long\def\dtv{d_{\textrm{\textup{TV}}}}%
\global\long\def\FI{\msf{FI}}%
\global\long\def\tv{\msf{TV}}%
\global\long\def\TV{\msf{TV}}%

\global\long\def\cov{\on{cov}}%
\global\long\def\var{\on{Var}}%
\global\long\def\ent{\on{Ent}}%

\global\long\def\cred#1{\textcolor{red}{#1}}%
\global\long\def\cblue#1{\textcolor{blue}{#1}}%
\global\long\def\cgreen#1{\textcolor{green}{#1}}%
\global\long\def\ccyan#1{\textcolor{cyan}{#1}}%
\global\long\def\yk#1{\textcolor{red}{\textsf{[YK: #1]}}}%
\global\long\def\yb#1{\textcolor{blue}{\textsf{[yb: #1]}}}%

\global\long\def\iff{\Leftrightarrow}%
 
\global\long\def\textfrac#1#2{{\textstyle \frac{#1}{#2}}}%

\global\long\def\Expo{\textnormal{Expo}}%
\global\long\def\Tr{\on{Tr}}%
\global\long\def\onu{\bar{\nu}}%
\global\long\def\intk{\inter\K}%
\global\long\def\ncal{\mathcal{N}}%
\global\long\def\svec{\operatorname{svec}}%
\global\long\def\tvec{\operatorname{vec}}%
\global\long\def\del{\partial}%
\global\long\def\ovec{\operatorname{ovec}}%

\global\long\def\Sph{\mathbb{S}}%
\global\long\def\gap{\operatorname{gap}}%
\global\long\def\Id{\mathrm{Id}}%
\global\long\def\HR{\mathrm{HR}}%
\global\long\def\CHAR{\mathrm{CHAR}}%
\global\long\def\logp{\log_{+}}%

%% file: _macros.tex
\def\balign#1\ealign{\begin{align}#1\end{align}}
\def\baligns#1\ealigns{\begin{align*}#1\end{align*}}
\def\balignat#1\ealign{\begin{alignat}#1\end{alignat}}
\def\balignats#1\ealigns{\begin{alignat*}#1\end{alignat*}}
\def\bitemize#1\eitemize{\begin{itemize}#1\end{itemize}}
\def\benumerate#1\eenumerate{\begin{enumerate}#1\end{enumerate}}

\newenvironment{talign*}
 {\let\displaystyle\textstyle\csname align*\endcsname}
 {\endalign}
\newenvironment{talign}
 {\let\displaystyle\textstyle\csname align\endcsname}
 {\endalign}

\def\balignst#1\ealignst{\begin{talign*}#1\end{talign*}}
\def\balignt#1\ealignt{\begin{talign}#1\end{talign}}

\let\originalleft\left
\let\originalright\right
\renewcommand{\left}{\mathopen{}\mathclose\bgroup\originalleft}
\renewcommand{\right}{\aftergroup\egroup\originalright}

\def\Gronwall{Gr\"onwall\xspace}
\def\Holder{H\"older\xspace}
\def\Ito{It\^o\xspace}
\def\Nystrom{Nystr\"om\xspace}
\def\Schatten{Sch\"atten\xspace}
\def\Matern{Mat\'ern\xspace}

\def\tinycitep*#1{{\tiny\citep*{#1}}}
\def\tinycitealt*#1{{\tiny\citealt*{#1}}}
\def\tinycite*#1{{\tiny\cite*{#1}}}
\def\smallcitep*#1{{\scriptsize\citep*{#1}}}
\def\smallcitealt*#1{{\scriptsize\citealt*{#1}}}
\def\smallcite*#1{{\scriptsize\cite*{#1}}}

\def\blue#1{\textcolor{blue}{{#1}}}
\def\green#1{\textcolor{green}{{#1}}}
\def\orange#1{\textcolor{orange}{{#1}}}
\def\purple#1{\textcolor{purple}{{#1}}}
\def\red#1{\textcolor{red}{{#1}}}
\def\teal#1{\textcolor{teal}{{#1}}}

\def\mbi#1{\boldsymbol{#1}} 
\def\mbf#1{\mathbf{#1}}
\def\mrm#1{\mathrm{#1}}
\def\tbf#1{\textbf{#1}}
\def\tsc#1{\textsc{#1}}

\def\mbiA{\mbi{A}}
\def\mbiB{\mbi{B}}
\def\mbiC{\mbi{C}}
\def\mbiDelta{\mbi{\Delta}}
\def\mbif{\mbi{f}}
\def\mbiF{\mbi{F}}
\def\mbih{\mbi{g}}
\def\mbiG{\mbi{G}}
\def\mbih{\mbi{h}}
\def\mbiH{\mbi{H}}
\def\mbiI{\mbi{I}}
\def\mbim{\mbi{m}}
\def\mbiP{\mbi{P}}
\def\mbiQ{\mbi{Q}}
\def\mbiR{\mbi{R}}
\def\mbiv{\mbi{v}}
\def\mbiV{\mbi{V}}
\def\mbiW{\mbi{W}}
\def\mbiX{\mbi{X}}
\def\mbiY{\mbi{Y}}
\def\mbiZ{\mbi{Z}}

\def\textsum{{\textstyle\sum}} 
\def\textprod{{\textstyle\prod}} 
\def\textbigcap{{\textstyle\bigcap}} 
\def\textbigcup{{\textstyle\bigcup}} 

\def\reals{\mathbb{R}} 
\def\integers{\mathbb{Z}} 
\def\rationals{\mathbb{Q}} 
\def\naturals{\mathbb{N}} 
\def\complex{\mathbb{C}} 

\def\what#1{\widehat{#1}}

\def\twovec#1#2{\left[\begin{array}{c}{#1} \\ {#2}\end{array}\right]}
\def\threevec#1#2#3{\left[\begin{array}{c}{#1} \\ {#2} \\ {#3} \end{array}\right]}
\def\nvec#1#2#3{\left[\begin{array}{c}{#1} \\ {#2} \\ \vdots \\ {#3}\end{array}\right]} 

\def\maxeig#1{\lambda_{\mathrm{max}}\left({#1}\right)}
\def\mineig#1{\lambda_{\mathrm{min}}\left({#1}\right)}

\def\Re{\operatorname{Re}} 
\def\indic#1{\mbb{I}\left[{#1}\right]} 
\def\logarg#1{\log\left({#1}\right)} 
\def\polylog{\operatorname{polylog}}
\def\maxarg#1{\max\left({#1}\right)} 
\def\minarg#1{\min\left({#1}\right)} 
\def\Earg#1{\E\left[{#1}\right]}
\def\Esub#1{\E_{#1}}
\def\Esubarg#1#2{\E_{#1}\left[{#2}\right]}
\def\bigO#1{\mathcal{O}\left(#1\right)} 
\def\littleO#1{o(#1)} 
\def\P{\mbb{P}} 
\def\Parg#1{\P\left({#1}\right)}
\def\Psubarg#1#2{\P_{#1}\left[{#2}\right]}
\def\Trarg#1{\Tr\left[{#1}\right]} 
\def\trarg#1{\tr\left[{#1}\right]} 
\def\Var{\mrm{Var}} 
\def\Vararg#1{\Var\left[{#1}\right]}
\def\Varsubarg#1#2{\Var_{#1}\left[{#2}\right]}
\def\Cov{\mrm{Cov}} 
\def\Covarg#1{\Cov\left[{#1}\right]}
\def\Covsubarg#1#2{\Cov_{#1}\left[{#2}\right]}
\def\Corr{\mrm{Corr}} 
\def\Corrarg#1{\Corr\left[{#1}\right]}
\def\Corrsubarg#1#2{\Corr_{#1}\left[{#2}\right]}
\newcommand{\info}[3][{}]{\mathbb{I}_{#1}\left({#2};{#3}\right)} 
\newcommand{\staticexp}[1]{\operatorname{exp}(#1)} 
\newcommand{\loglihood}[0]{\mathcal{L}} 


\providecommand{\arccos}{\mathop\mathrm{arccos}}
\providecommand{\dom}{\mathop\mathrm{dom}}
\providecommand{\diag}{\mathop\mathrm{diag}}
\providecommand{\tr}{\mathop\mathrm{tr}}
\providecommand{\card}{\mathop\mathrm{card}}
\providecommand{\sign}{\mathop\mathrm{sign}}
\providecommand{\conv}{\mathop\mathrm{conv}} 
\def\rank#1{\mathrm{rank}({#1})}
\def\supp#1{\mathrm{supp}({#1})}

\providecommand{\minimize}{\mathop\mathrm{minimize}}
\providecommand{\maximize}{\mathop\mathrm{maximize}}
\providecommand{\subjectto}{\mathop\mathrm{subject\;to}}

\def\openright#1#2{\left[{#1}, {#2}\right)}

\ifdefined\nonewproofenvironments\else
\ifdefined\ispres\else
 
\fi
\fi
\makeatletter
\@addtoreset{equation}{section}
\makeatother
\def\theequation{\thesection.\arabic{equation}}

\newcommand{\cmark}{\ding{51}}

\newcommand{\xmark}{\ding{55}}

\newcommand{\eq}[1]{\begin{align}#1\end{align}}
\newcommand{\eqn}[1]{\begin{align*}#1\end{align*}}
\newcommand{\Ex}[1]{\mathbb{E}\left[#1\right]}